\documentclass[11pt]{amsart}

\usepackage{amsmath,amssymb,latexsym}
\usepackage{amsfonts,amstext,amsthm,amscd}
\usepackage{graphicx}
\usepackage{mathrsfs}
\usepackage{hyperref}
\theoremstyle{plain}
\usepackage{color}
\usepackage[all]{xy}
\usepackage[top=1in, bottom=1in, left=1.2in, right=1.2in]{geometry}
\usepackage{dsfont}

\newtheorem{theorem}{Theorem}[section]

\newtheorem{lemma}[theorem]{Lemma}

\newtheorem{proposition}[theorem]{Proposition}

\newcommand{\CC}{\mathbb{C}}

\newcommand{\ZZ}{\mathbb{Z}}

\newcommand{\fc}{f_c}
\newcommand{\Fc}{F_c}

\usepackage[draft]{changes}
\definechangesauthor[color=blue]{jk}
\definechangesauthor[color=red]{ae}

\def\C{\mathbb{C}}

\def\R{\mathbb{R}}

\def\N{\mathbb{N}}

\def\Z{\mathbb{Z}}
\def\D{\mathbb{D}}

\def\fL{\mathfrak{L}}

\def\CP2{{\mathbb{CP}^2}}

\def\real{{\operatorname{Re}}}
\def\ima{{\operatorname{Im}}}
\def\dist{{\operatorname{dist}}}
\def\ps{{\operatorname{PS}}}

\def\esc{{\operatorname{I}_\infty}}
\def\Esc{{\operatorname{I}_\infty}}
\def\hh{{\operatorname{H}_\alpha}}

\def\HD{\operatorname{HD}}
\def\h1{\operatorname{H}_1}
\def\hsin{\operatorname{H}}

\usepackage{calrsfs}
\DeclareMathAlphabet{\pazocal}{OMS}{zplm}{m}{n}

\def\cC{{\pazocal{C}}}
\def\cD{{\pazocal{D}}}

\def\cF{{\pazocal{F}}}

\def\cJ{{\pazocal{J}}}

\def\cL{{\pazocal{L}}}

\def\cP{{\pazocal{P}}}

\def\cS{{\pazocal{S}}}

\def\cU{{\pazocal{U}}}

\def\hK{{\hat{K}}}

\def\fatou{\cF}
\def\julia{\cJ}

\def\fc{f_{\mathbf{c}}}

\begin{document}
\title{Thermodynamic formalism and hyperbolic Baker domains II: \\ real-analyticity of the Hausdorff dimension}
\author{Adri\'an Esparza-Amador} 
\address{Instituto de Ciencias Físicas y Matemáticas, Universidad Austral de Chile}
\email{adrian.esparza@uach.cl}

\date{\today}
\maketitle
\begin{abstract}
We consider the family of entire maps given by $f_{\ell,c}(z)=\ell+c-(\ell-1)\log c-e^z$, where $c\in\cD(\ell,1)$ and $\ell\in\N$, $\ell\geq2$. By using the property of $\fc$ to be dynamically projected to an infinite cylinder $\CC/2\pi\ZZ$ where the thermodynamic formalism tools are well-defined, we prove as a main result on this work, the real-analyticity of the map $c\mapsto \HD(\julia_r(\fc))$, here $\julia_r(\fc)$ is the \emph{radial Julia set}. 
\end{abstract}

\section{Introduction}
For a transcendental entire map $f:\C\to\C$, the \emph{Fatou set}, denoted by $\fatou(f)$, is defined as the set of points $z\in\C$ for which there exists a neighbourhood $U$ where the iterates $\{f^n\big|_U\}_{n\geq0}$ are normal in the sense of Montel. Its complement $\julia(f)=\C\setminus\fatou(f)$, called the \emph{Julia set} of $f$, is characterize by containing the \emph{chaotic dynamical} part of the map $f$. Moreover, it is a well known result from Baker that repelling periodic points are dense in the Julia set, \cite{baker1968}. 

Given an integer $\ell\geq2$, for $c\in\cD(\ell,1)$ we consider the one-parameter family of entire maps 
\[
f_c(z)=\ell z+c-(\ell-1)\log c - e^z.
\]
By noting that for $\ell=c=2$, the function $f_2$ corresponds to the one studied by Bergweiler in \cite{berg1995}, we named this as the \emph{Bergweiler's family}. Following Bergweiler's ideas, authors in \cite{EI2023} gave a general description of the Fatou set of $f_c$, see Theorem \ref{Fatou_set} below. Although $\fatou(f_c)$ contains a hyperbolic Baker domain in the left-half plane, and wandering domains in a vertical strip, is it possible to prove that $f_c$ is \emph{E-hyperbolic} \cite[Theorem 1]{EI2023}, which is an extension of the comcept of hyperbolicity (or expanding) of rational maps to transcendental maps, also referred as \emph{topologically hyperbolic}, see \cite{sta90} and \cite{mu2021} for further references. 

Thermodynamic formalism has proven to be an appropriate theory to provide probabilistic description of the chaotic dynamical part of a dynamical system. In the context of holomorphic expanding/hyperbolic dynamical systems gives a rich and detailed information about the geometry of Julia sets, see \cite{bar1995}, \cite{ku2002}, \cite{mu2008}, and the general survey \cite{mu2021} and the references therein. 

Following ideas of Kotus and Urba\'nski in \cite{KU05}, the authors in \cite{EI2023} studied the \emph{radial Julia set} $\julia_r(\fc)$, of each member of the family $f_c$ obtaining important properties of the geometry of this set by establishing the existence and uniqueness of Gibbs and equilibrium states, studying spectral and asymptotic properties of corresponding Perron-Frobenius operators, showing that dynamical systems are "strongly" mixing, and satisfy the Central Limit Theorem. Finally, as a main result, they obtained a Bowen's formula for the Hausdorff dimension of the radial Julia set, $\HD(\julia_r(f_c))$. 

The aim of this paper is to give a continuation of the previous work in \cite{EI2023}, in particular, to prove the real-analyticity of the Bowen's formula for the map $c\mapsto\HD(\julia_r(f_c))$.  We now briefly describe the organization of this paper. 

In Section \ref{sec_fam}, the dynamical description of the family $f_c$ and its Fatou set $\fatou(f_c)$ is given as well as the definition of the projected map $F_c:Q\to Q$ to the infinite cylinder $Q=\C/2\pi i\Z$, which is the central object of study. In Section \ref{sec_form}, we define the \emph{transfer operator} and enumerate the results on thermodynamic formalism from \cite{EI2023} that will be useful in the main theorem. Section \ref{sec_qc} contains \emph{quasi-conformal} facts of the family $F_c$ which imply its Hölder (analyticity) regularity. Finally, in Section \ref{sec_an} we prove the real-analyticity of the Hausdorff dimension of the radial Julia set. 

\section{The Bergweiler's family}\label{sec_fam}
In this section we give a general description of the dynamical plane of each member of the family of entire maps considered in this work. 

For $\ell\in\N$, $\ell\geq2$, take $c\in\cD(\ell,1)$ and consider the entire map given by 
\begin{equation}\label{def_fc}
f_c(z):=f_{\ell,c}(z) = \ell z+c-(\ell-1)\log c-e^z. 
\end{equation}
We notice that the special parametrization of the family $f_c$, allows us to conclude that $z_c=\log c$ is a fixed point. Since 
\begin{equation}\label{f_prime}
f_c'(z) = \ell - e^z, 
\end{equation}
its multiplier is given by 
\[
\lambda_c = \ell -c, 
\]
which concludes that $z_c=\log c$ is a (super-)attracting fixed point for $f_c$. Furthermore, from (\ref{def_fc}) we notice $f_c$ has no asymptotic values, and from (\ref{f_prime}) their critical points are 
\[
\text{Crit}(f_c) = \{\log \ell + 2\pi ik:k\in\Z\}.
\]
Since $f_c(z+2\pi ik) = f_c(z)+\ell(2\pi ik)$, the singular values of $f_c$ are given by 
\begin{equation}
\text{sing}(f_c^{-1})=f_c(\text{Crit}(f_c))=\{f_c(\log \ell)+\ell(2\pi ik):k\in\Z\}. \\
\end{equation}
We are in position to give a general description of the Fatou set of each member $f_c$ of the family, see \cite{EI2023} for further details. 
\begin{theorem}[Corollary 1, \cite{EI2023}]\label{Fatou_set}
Fix $\ell\in\N$, $\ell\ge2$. For $c\in\cD(\ell,1)$ the Fatou set $\cF(f_c)$ of the member $f_c$ as defined in (\ref{def_fc}) consists of the following sets: 
\begin{itemize}
\item A univalent invariant hyperbolic Baker domain $\cU_\ell$ containing the half-plane $\{\real (z)<-2\ell\}$ and all of its preimages. 
\item A simply connected (Böttcher) Schröeder domain $\cU_0$ containing the (super-)attracting fixed point $z_c=\log c$ and all of its preimages. 
\item Two sets of wandering domains from the $2k\pi i$-traslation of the attracting domain $\cU_0$ and all of its preimages. 
\end{itemize}
\end{theorem}

It is easy to see from the above result, that the Euclidean distance holds
$$
\dist(\cP(f_c),\julia(f_c))>0,
$$
where $\cP(f_c)$ denotes the post-singular set of $f_c$ define as 
$$
\cP(f_c)=\overline{\bigcup_{n\geq0}f_c^n(\text{sing}(f_c^{-1}})).
$$
Then $f_c$ belongs to the class $C$ described by Stallard in \cite{sta90}, also called E-hyperbolic or topologically hyperbolic class, see \cite{mu2021} for further references. It follows that the Lebesgue measure of the Julia set of $f_c$ is zero, \cite[Theorem B]{sta90}. 

\subsection{Projection on the cylinder $\C/\sim$}
Consider the infinite cylinder $Q=\C/\sim$ where $z\sim w$ if and only if $z-w\in 2\pi i\Z$. The quotient space $\C/\sim$ is provided with a Riemann surface structure from the quotient map 
\begin{displaymath}
\begin{array}{rccc}
\Pi: & \C & \rightarrow & Q \\
     & z & \mapsto & [z].
\end{array}
\end{displaymath}
The relation $f_c(z+2\pi ik) = f_c(z)+\ell(2\pi ik)$ means that $f_c$ respects the equivalence relation $\sim$, so it induces a unique map 
\[
\Fc:Q\to Q
\]
such that $\Pi\circ \fc = \Fc\circ \Pi$. Following \cite{berg19952} we know that $\julia(F_c)=\Pi(\julia(f_c))$, so from now on, we focus on $F_c$ instead of the original function $f_c$. Using Theorem \ref{Fatou_set} we can characterize the dynamics under $F_c$ of any point in the Fatou set $\fatou(F_c)$. 

\begin{proposition}
Given $z\in\fatou(F_c)$, there are only two options for the limit of the sequence $\{z_n=F_c^n(z)\}_n^\infty$ either 
$$
\lim_{n\to\infty}z_n=\log c,
$$
or
$$
\lim_{n\to\infty}z_n=-\infty. 
$$
\end{proposition}

On the other hand, since $\fc$ is E-hyperbolic, it can be proved that $\Fc$ is an expansive function. 
\begin{proposition}[Proposition 1, \cite{EI2023}]\label{prop1}
There exist $L>0$ and $\kappa>1$ such that for every $z\in\julia(\Fc)$ and every $n\ge1$ we have
\[
|(\Fc^n)'(z)|\geq L\kappa^n.
\]
\end{proposition} 

We consider the escaping subset on $\julia(\Fc)$ given by 
\[
\Esc(\Fc) = \{z:\in\julia(\Fc):\lim_{n\to\infty}\Fc^n(z)=+\infty\},
\]
and its corresponding set on $\julia(\fc)=\Pi^{-1}(\julia(\Fc))$ as 
\[
\esc(\fc)=\{w\in\julia(\fc):\Pi(w)\in\Esc(\Fc)\}.
\]
Given the ``asymptotic'' similarity to the exponential function $z\mapsto -e^z$, it is not difficult to notice that $\HD(\esc(F))=\HD(\esc(f_c))=\HD(\julia(F_c))=\HD(\julia(f_c))=2$, which makes these sets dynamically rather boring. 

We denote by $\julia_r(\Fc)$, called the \emph{radial Julia set}, to the complement on $\julia(\Fc)$ of the escaping set, 
\[
\cJ_r(\Fc)=\cJ(\Fc)\setminus \Esc(\Fc)\ \text{ and }\ \cJ_r(\fc)=\cJ(\fc)\setminus \esc(\fc).
\]
This way, the thermodynamic formalism will be developed on this rather interesting set \cite[Lemma 4]{mu2021}.

\section{Thermodynamic formalism}\label{sec_form}
In the present section, we present the most important properties in \cite{EI2023} derived from the thermodynamic formalism that will be useful in the rest of the paper.
\subsection{Transfer Operator and Topological Pressure}
Let $\cC_b=\cC_b(\julia(\Fc))$ be the Banach space of all bounded continuous complex valued functions on $\julia(\Fc)$. For $t>0$, define the (transfer) Perron-Frobenius operator $\cL_t:\cC_b\to\cC_b$ given by 
\begin{equation}\label{transfer_op}
\cL_tg(z) = \sum_{x\in\Fc^{-1}(z)}|\Fc'(x)|^{-t}g(x). 
\end{equation}
For any $n\geq1$, its $n$-iterate is defined as 
\begin{equation}\label{ite_transfer}
\cL_t^n g(z) = \sum_{x\in\Fc^{-n}(z)}|(\Fc^n)'(x)|^{-t}. 
\end{equation}
We use the transfer operator and its iterates to define the \emph{topological pressure} of the (non-compact) Julia set of $F$. First, for every $t>0$ and every $z\in Q\setminus \cP(\Fc)$, we define the \emph{lower} and \emph{upper topological pressure} as 
\[
\underline{P}(t,z) = \liminf_{n\to\infty}\dfrac{1}{n}\log \cL_t^n\mathds{1}(z)\ \text{ and }\ \overline{P}(t,z)=\limsup_{n\to\infty}\dfrac{1}{n}\log \cL_t^n\mathds 1(z).
\]

\begin{proposition}[Proposition 2, \cite{EI2023}]
Let $t>1$ and $z\in\julia(\Fc)$, then the following statements hold:
\begin{itemize}
\item $\underline{P}(t,z)$ and $\overline{P}(t,z)$ are independent on the choice of $z\in\julia(\Fc)$. 
\item $\|\cL_t\mathds 1\|_\infty = \sup\{\cL_t\mathds 1(z):z\in\julia(\Fc)\}<+\infty$.
\item For every $t>1$, $\underline{P}(t)$ and $\overline{P}(t)$ are convex, continuous and strictly decreasing functions and $\lim_{t\to\infty}\overline{P}(t)=-\infty$. 
\item $\lim_{\real(z)\to+\infty}\cL_t\mathds 1(z)=0.$
\end{itemize}
\end{proposition}

\subsection{Conformal measures} Let $t,\alpha\in\R$. We say that a measure $\nu$ on $\julia(\Fc)$ is $(t,\alpha)$-\emph{conformal} if for every Borel set $A\subset \julia(\Fc)$ where $\Fc\big|_A$ is injective, we have 
\[
\nu(\Fc(A))=\int_A\alpha|\Fc'|^td\nu.
\]

\begin{theorem}[Theorem 3, \cite{EI2023}]
For every $t>1$, there exists $\alpha_t>0$ and a $(t,\alpha_t)$-conformal measure $m_t$ for the map $F:\julia(\Fc)\to\julia(\Fc)$ with $m_t(\julia(\Fc))=1$. 
\end{theorem}
Now, let $\cL_t^*:\cC_b(\julia(\Fc))^*\to\cC_b(\julia(\Fc))^*$ denote the dual operator of $\cL_t$ given by the formula 
\[
\int\phi d(\cL_t^*\mu) = \int\cL_t\phi d\mu,\qquad \forall \phi\in\cC_b(\julia(\Fc)).
\]
Set $\widehat{\cL_t}=\alpha_t^{-1}\cL_t$ and let $\widehat{\cL_t^*}$ denote its dual. 
\begin{proposition}[Proposition 4, \cite{EI2023}]\label{prop_press}
Let $t>1$, the following properties hold: 
\begin{itemize}
\item[(i)] $\cL_t^*m_t=\alpha_t m_t$ and $\widehat{\cL_t^*}m_t=m_t$. 
\item[(ii)] $\sup_{n\geq1}\{\|\widehat{\cL_t^n}\mathds 1\|_\infty\}<+\infty$. 
\item[(iii)] $P(t)=\underline{P}(t)=\overline{P}(t)=\log\alpha_t$.  
\end{itemize}
\end{proposition}

The commom value $P(t)$ in $(iii)$ above is called the \emph{topological pressure} of $F$. 

\begin{lemma}[Lemma 3, \cite{EI2023}]
If $\nu$ is a $(t,\beta^j)$-conformal measure for $F^j$ with $t>1$, then there exists $M>0$ such that 
\[
\liminf_{n\to\infty}\real(\Fc^{nj}(x))\leq M, 
\]
for $\nu$-a.e. $x\in\cJ(\Fc)$. In particular, $\nu(\esc(\Fc))=0$ so $\nu(\cJ_r(\Fc))=1$. 
\end{lemma}
The above result confirms that $\cJ_r$ is the right subset on $\cJ(\Fc)$ to be geometrically studied. As a first consequence we have. 

\begin{theorem}[Theorem 4, \cite{EI2023}]
For $t>1$, the probability measure $m_t$ is the unique $(t,e^{P(t)})$-conformal measure for $\Fc$. Moreover, it is ergodic with respect each iterate $\Fc^j$. 
\end{theorem}
Finally, in order to state the Bowen's formula for $\cJ_r$, we need to define the bounded variation Banach space on which our transfer operator will be acting. 

For a fixed $\alpha\in(0,1]$ and $g\in\cC_b$ we define the $\alpha$-variation of $g$ given by 
\[
v_\alpha(g)=\inf\{L\geq0:|g(y)-g(x)|\leq L|y-x|^\alpha,\ x,y\in\cJ(\Fc),\ \text{ with }|y-x|\leq\delta\},
\]
where $\delta=1/2\min\{1/2,\dist(\cP(\Fc),\cJ(\Fc))\}>0$, and set 
\[
\|g\|_\alpha = v_\alpha(g) + \|g\|_\infty.
\]
It follows that the normed space 
\[
(\hh,\|\cdot\|_\alpha):=\{g\in\cC_b:\|g\|<+\infty\}
\]
is a Banach space densely contained in $\cC_b$ with respect to the $\|\cdot\|_\infty$-norm. We have the following result, when restrict the transfer operator to the Banach space $\hh$. 

\begin{theorem}[Theorem 5, \cite{EI2023}]\label{theo5.4}
For $t>1$, we have the following: 
\begin{itemize}
\item[(a)] The number 1 is a simple isolated eigenvalue of $\widehat{\cL_t}:\hh\to\hh$. 
\item[(b)] The eigenspace of the eigenvalue 1 is generated by nowhere vanishing function $\psi_t\in\hh$ such that $\int\psi_t dm_t=1$ and 
$$\lim_{\real(z)\to +\infty}\psi_t(z)=0.$$
\item[(c)] The number 1 is the only eigenvalue of modulus 1. 
\end{itemize}
\end{theorem}
We then have the Bowen's formula for $\cJ_r$. 
\begin{theorem}[Theorem 8, \cite{EI2023}]\label{bowen}
The Hausdorff dimension of the radial Julia set $\cJ_r(\Fc)$ is the unique zero of the pressure function $t\mapsto P(t)$, $t>1$. 
\end{theorem}

\subsection{Analyticity of the Perron-Frobenius operator}
Recall that, since $\dist(\cP(\Fc),\cJ(\Fc))>0$, we have 
\[
\delta = \dfrac{1}{2}\min\{1/2,\dist(\ps(\Fc),\cJ(\Fc))\}>0. 
\]
For every $n\geq1$ and every $v\in\cJ(\Fc)$, let 
\[
(\Fc)_v^{-n}:B(\Fc^n(v),2\delta)\to Q,
\]
denote the holomorphic inverse branch of $\Fc^n$ defined on the open ball $B(\Fc^n(v),2\delta)$ and sending $F_c^n(v)$ to $v$. It follows from Proposition \ref{prop1} and Koebe's distorsion theorem that there exist constants $L>0$ and $0<\beta<1$ such that for every $n\geq1$ and everu $v\in\cJ(\Fc)$ and every $z\in B(\Fc^n(v),2\delta)$ we have 
\begin{equation}\label{eq1*}
|((\Fc)_v^n)'(z)|\leq L\beta^n.
\end{equation}

We say that a continuous function $\phi:\cJ(\Fc)\to\C$ is dynamically $\alpha$-Hölder if there exists $c_\phi>0$ such that 
\begin{equation}\label{Holder_alp}
|\phi_n((\Fc)_v^{-n}(y))-\phi_n((\Fc)_v^{-n}(x))|\leq c_\phi|\phi_n((\Fc)_v^{-n}(x))||x-y|^\alpha,
\end{equation}
for all $n\geq1$, all $x,y\in\cJ(\Fc)$ with $|x-y|\leq\delta$ and all $v\in \Fc^{-n}(x)$, where 
\[
\phi_n(z)=\phi(z)\phi(F(z))\cdots \phi(\Fc^{n-1}(z)).
\]
And we say that is summable if 
\[
\sup_{z\in\cJ(\Fc)}\left\{\sum_{v\in\Fc^{-1}(z)}\|\phi\circ(\Fc)_v^{-1}\|_\infty\right\}<+\infty.
\]
We end this section by collecting the main result of Section 8 in \cite{KU05}. We denote the class of $\alpha$-Hölder continuous summable functions on $\cJ(\Fc)$ by $\hsin_\alpha^s$. 

\begin{theorem}[Corollary 8.7, \cite{KU05}]\label{cor8.7}
Suppose that $G$ is an open connected subset of $\C^n$, $n\geq1$, and that $\phi_\sigma:\cJ(\Fc)\to\C$, $\sigma\in G$, is a family of mappings such that the following assumptions are satisfied. 
\begin{itemize}
\item[(a)] For every $\sigma\in G$, $\phi_\sigma$ is in $\hsin_\alpha^s$. 
\item[(b)] For every $\sigma\in G$, the function $\phi_\sigma$ is dynamically $\alpha$-Hölder. 
\item[(c)] The function $\sigma\mapsto\phi_\sigma\in\hh$ ($\sigma\in G$) is continuous. 
\item[(d)] The family $\{c_{\phi_\sigma}\}_{\sigma\in G}$ is bounded. 
\item[(e)] The function $\sigma\mapsto\phi_\sigma(z)\in\C$, $\sigma\ G$, is holomorphic for every $z\in\cJ(\Fc)$. 
\item[(f)] $\forall(\sigma_2\in G)\exists(r>0)\exists(\sigma_1\in G)$ such that 
\[
\sup\left\{\left|\dfrac{\phi_\sigma}{\phi_{\sigma_1}}\right|:\sigma\in\overline{B(\sigma_2,r)}\right\}<+\infty.
\] 
Then the function $\sigma\mapsto\cL_{\phi_{\sigma}}\in L(\hsin_\alpha)$, $\sigma\in G$, is holomorphic. 
\end{itemize}
\end{theorem}

\section{Quasi-conformal Conjugacies of the Family $\{\fc\}$}\label{sec_qc}
The following continuity result on the Julia set will be useful in the analyticity of the Hausdorff dimension. 
\begin{lemma}
Let $c_0, c_n\in\cD(\ell,1)$ with $c_n\to c_0$ as $n\to\infty$. If $z_n\in \cJ(F_{c_n})$ for all $n\geq1$, and if $\lim_{n\to\infty} z_n=z_0$ for some $z_0\in\C$, then $z\in\cJ(F_{c_0})$. 
\end{lemma}
\begin{proof}
From Theorem \ref{Fatou_set} we know that the Fatou set $\cF(\Fc)$ consists only of the Baker (strip) domain $\Pi(\cU_\ell)$, the attracting domain $\Pi(\cU_n)$ and all of their preimages for each $c\in\cD(\ell,1)$. Following ideas in \cite{EM2021} we conclude that the mapping $c\mapsto\cJ(\Fc)$, $c\in\cD(\ell,1)$ is continuous in the Hausdorff metric, which finishes the proof of the lemma. 
\end{proof}

We say that an entire function $f$ belongs to the class $S_q$ if the set $\text{sing}(f^{-1})$ contains at most $q$ points. We set $\cS=\bigcup_{q=1}^{\infty} S_q$, commonly call the \emph{Speiser class}, see \cite{el1992} for further references. 

Fix $g\in S_q$ and define $M_g\subset S_q$ as
$$
M_g = \{f\in S_q: \exists \varphi,\psi:\C\to\C\ \text{ such that } \psi\circ g=f\circ\varphi\}.
$$
Consider the multi-valued analytic function $\alpha_p:M_g\to\C$ satisfying the equation $f^p(\alpha)=\alpha$. Let $N_p$ denote the algebraic singularities (\cite[Theorem 2]{el1992}) of this function and set $N=\overline{\bigcup_{p=1}^\infty}N_p$ and $\Sigma=M\setminus N$. 

\begin{theorem}[Theorem 10, \cite{el1992}]\label{el92_theo}
Each $f\in M_g$ is structurally stable and the conjugating homeomorphisms can be chosen to be quasi-conformal. 
\end{theorem}

Now, the two covering maps $\Pi:\C\to Q=\C/\sim$ and the map $z\mapsto e^z\in\C^*=\C\setminus\{0\}$ induce a conformal homeomorphism $H:Q\to\C^*$ which extends to a conformal homeomorphism $H:Q\to\C^*$ sending $-\infty$ to $0$. Each map $G_c=H\circ\Fc\circ H^{-1}:\C\to\C$ is given by 
\begin{equation}\label{eq9.1}
G_c(z)=\dfrac{1}{c^{\ell-1}}z^\ell e^{c-z}.
\end{equation}

Assuming that $c\in\cD(\ell,1)$, the map contains exactly one super-attracting fixed point in $z_0=0$ and a (super-) attracting fixed point at $z=c$. $G_c$ has only two singularities at $z_0=0$ and $z=\ell$, but $z_0=0$ is a super-attracting fixed point for $G_c$. $z=c\in\cD(\ell,1)$ implies that the critical point $z=\ell$ belongs to the Fatou set of $G_c$, $\cF(G_c)$, and is attracted by the fixed point $z=c$ under forward iterates of $G_c$. 

In particular $G_c$ belongs to the so-called Speiser class $\cS$ as described above, and we have that $G_c\in M=M_{G_\ell}$. Moreover, $\{G_c\}\subset\Sigma$, $c\in\cD(\ell,1)$. Following Theorem \ref{el92_theo} we have
\begin{theorem}\label{prop9.2}
Fix $c_0\in\cD(\ell,1)$. Then for every $c\in\cD(\ell,1)$ there exists a quasi-conformal homeomorphism conjugating $G_c$ and $G_{c_0}$ (i.e. $G_c\circ H_c=H_c\circ G_{c_0}$, with $H_c$-q.c.). These conjugating homeomorphisms can be chosen so that the following properties are satisfied. 
\begin{itemize}
\item[(a)] For every $z\in\C$, the map $c\mapsto H_c(z)$ is holomorphic. 
\item[(b)] The mapping $(c,z)\mapsto H_c(z)$ is continuous. 
\item[(c)] The dilatation of the maps $H_c$ converge to 1 when $c\to c_0$. 
\end{itemize}
\end{theorem}
Since each map $F_c$ is conjugated with $G_{c_0}$ by the same conformal homeomorphism $H:Q\to\C^*$, we have the following extension. 
\begin{theorem}\label{prop9.3}
Fix $c_0\in\cD(\ell,1)$. Then for every $c\in\cD(\ell,1)$, there exists a quasi-conformal homeomorphism conjugating $F_c$ with $F_{c_0}$ (i.e. $F_c\circ h_c=h_c\circ F_{c_0}$, with $h_c$-q.c.). These conjugating homeomorphims can be chosen so that the following properties are satisfied. 
\begin{itemize}
\item[(a)] For every $z\in Q$, the map $c\mapsto h_c(z)$ is holomorphic. 
\item[(b)] The mapping $(c,z)\mapsto h_c(z)$ is continuous. 
\item[(c)] The dilatation of the maps $h_c$ converges to 1 when $c\to c_0$. 
\end{itemize}
\end{theorem}

The following result tell us that the expansiveness in Proposition \ref{prop1} is, in some sense, continuous. 

\begin{lemma}\label{lemma9.4}
For $c_0\in\cD(\ell,1)$ there exist constants $r>0$, $L>0$ and $\kappa>1$ such that 
\[
|(\Fc^n)'(z)|\geq L\kappa^n,
\]
for all $c\in \cD(c_0,r)$, all $z\in\cJ(\Fc)$ and all $n\geq1$. 
\end{lemma}
\begin{proof}
We consider the following notation. For $M>0$, 
\[
Q_M=\{z\in Q:-2\ell\leq\real(z)\leq M\}. 
\]
Now, for every $c\in\cD(\ell,1)$ and every $k\geq1$, set 
\[
A_k(c)=\{z\in\C:|(\Fc^k)'(z)|>2\}.
\]
Since $\fc'(z)=\ell-e^z$, it follows that there exists $M>0$ such that 
$$Q_M^c\subset A_1(c),$$
for all $c\in\cD(\ell,1)$. Given the description of the Fatou set $\fatou(\fc)$ in Theorem \ref{Fatou_set}, it is easy to see that $\julia(F_{c_0})\setminus Q_M^c$ is a compact set and, from definition, all the sets $A_k(c_0)$, $k\geq1$ are open sets, then there exists $q\geq1$ such that $\julia(F_{c_0})\setminus Q_M^c\subset A_1(c_0)\cup...\cup A_q(c_0)$. Thus
\[
\julia(F_{c_0})\subset A_1(c_0)\cup...\cup A_q(c_0). 
\]
Again, by the Hausdorff continuity of $c\mapsto\julia(\Fc)$, it follows that there exists $r>0$ such that 
\begin{equation}\label{fin_cov}
\julia(F_{c})\subset A_1(c)\cup...\cup A_q(c),
\end{equation}
for $c\in B(c_0,r)$. But $\julia(\Fc)\subset Q$ is closed without critical points, hence 
\[
u = \min\{1,\inf\{|\Fc'(z)|:z\in\julia(\Fc)\}\}>0,
\]
since $\fc'(z)=\ell-e^z$. Fix $z\in\julia(\Fc)$ and $n\geq1$. Using (\ref{fin_cov}) we divide the finite sequence $\{\Fc^j(z)\}_{j=0}^n$ into blocks of length $L_0\leq q$ such that the modulus of the derivative of the composition along each such a block is larger than 2. This implies that $|(\Fc^n)'(z)|>2^{\lfloor n/q\rfloor}u^q$ which finishes the proof. 
\end{proof}

From Theorem \ref{prop9.3} we know that the map $c\mapsto h_c(z)$ is holomorphic for every $z\in Q$, so let $h'_c(z)$ denote the derivative of such map. We also consider the following notation for a clearer development in the proofs of the following results. We will writ $F(c,z)=\Fc(z)$ and more generally $F^n(c,z)=\Fc^n(z)$. 
\begin{proposition}\label{prop9.5}
For every $c_0\in\cD(\ell,1)$ there exists $r>0$ such that 
\[
T:=\sup\{|h_c'(z)|:c\in B(c_0,r),\ z\in\julia(F_{c_0})\}<+\infty.
\]
\end{proposition}
\begin{proof}
Let $z_0\in\julia(F_{c_0})$ be a periodic point of $F_{c_0}$. Fix a period $n\geq1$ of $z_0$, that is 
$$F^n(c,h_c(z_0))=h_c(z_0),$$
for all $c\in\cD(\ell,1)$. Now, differentiate the above equation with respect to the first variable $c$: 
\begin{equation}\label{eq9.7}
D_1F^n(c,h_c(z_0)) + D_2F^n(c,h_c(z_0))\cdot h_c'(z_0)=h_c'(z_0),
\end{equation}
which implies that 
\[
h_c'(z_0) = \dfrac{D_1F^n(c,h_c(z_0))}{1-D_2F^n(c,h_c(z_0))}.
\]
Now, from the expansiveness in Proposition \ref{prop1} we have that if $n\geq1$, as a period of $z_0$, is large enough (depending on $c$), then 
\[
|h_c'(z_0)|\leq2\dfrac{|D_1F^n(c,h_c(z_0))|}{|D_2F^n(c,h_c(z_0))|}. 
\]
So, to finishes the proof it is enough to get an upper bound for the right hand of the above inequality. For all $c\in\cD(\ell,1)$ and all $z\in\C$ we have 
\begin{align*}
D_1F^n(c,z) & = D_1(F(c,F^{n-1}(c,w))) \\
	& = D_1F(c,F^{n-1}(c,z)) + D_2F(c,F^{n-1}(c,z))D_1F^{n-1}(c,z) \\
	& = 1 + D_2F(c,F^{n-1}(c,z))D_1F^{n-1}(c,z),
\end{align*}
This way, inductively, we have 
\[
D_1F^n(c,z) = 1 +\sum_{j=0}^{n-1}D_2F^{n-j}(c,F^j(c,z)). 
\]
Hence, combining Lemma \ref{lemma9.4} with Equation (\ref{eq9.7}) we have 
\begin{equation}\label{eq9.8}
|h_c'(z)|\leq2\sum_{j=0}^n|(\Fc^j)'(h_c(z))|^{-1}\leq2\sum_{j=0}^\infty|(\Fc^j)'(h_c(z))|^{-1}\leq2 L^{-1}\sum_{j=0}^\infty\kappa^j=\dfrac{2\kappa}{L(\kappa-1},
\end{equation}
for all $c\in \cD(c_0,r)$, where $r>0$ is taken from Lemma \ref{lemma9.4}. Now, let $w\in\julia(F_{c_0})$ be an arbitrary point. Since periodic point of $F_{c_0}$ are dense in $\julia(F_{c_0})$, there exists a sequence of periodic points $\{w_n\}_n\subset\julia(F_{c_0})$ such that $\lim_n w_n=w$. From the continuity part Theorem \ref{prop9.3} (b), it follows that $\{c\mapsto h_c(w_n)\}\to (c\mapsto h_c(w))$ uniformly on all compact subsets of $\cD(\ell,1)$. From the holomorphic part Theorem \ref{prop9.3} (a), we have that the functions $c\mapsto h_c(\zeta)$, $\zeta\in\julia(F_{c_0})$ are analytic, combining all of these with Equation (\ref{eq9.8}) we conclude that 
\[
|h_c'(w)|=\lim_n|h_c'(w_n)|\leq\dfrac{2\kappa}{L(\kappa-1)}<+\infty.
\]  
\end{proof}
We are ready to state the main result on quasi-conformality that will be useful in the analyticity of the Hausdorff dimension. 

Given $K,\alpha>0$ and $c\in\cD(\ell,1)$ we say that a map $h:\C\to\C$ is $(K,\alpha,c)$-Hölder continuous if 
\[
|h(x)-h(y)|\leq K|x-y|^\alpha,
\]
for all $z\in\julia(\Fc)$ and all $x,y\in B(z,2^{-10})$. 

\begin{proposition}\label{prop9.6}
For every $c_0\in\cD(\ell,1)$ there exists $\hat{K}_{c_0}>1$ such that if $c\in \cD(c_0,r)$, with $r>0$ sufficiently small, then the conjugating homeomorphism $h_c:\C\to\C$ is $(\hat{K}_{c_0},1/q_c,c_0)$-Hölder continuous, where $q_c$ is the quasi-conformality constant of $h_c$. 
\end{proposition}
\begin{proof}
Combining Theorem \ref{prop9.3} and Proposition \ref{prop9.5} we may assume $r>0$ small enough so that $q_c<2$ for every $c\in\cD(c_0,r)$. Fix $x\in\julia(F_{c_0})$, and let $V=h_c(B(x,1))$. It follows that $V$ is an open simply connected set so let $\varphi:\D\to V$ be its conformal representation with $R(0)=h_c(x)$. It is easy to see from Theorem \ref{Fatou_set} that $\julia(\Fc)$ is connected for every $c\in\cD(\ell,1)$, so there exists a connected subset of $\julia(F_{c_0})$ joining $x$ and the boundary $\partial B(x,1)$. From definition $\julia(\Fc)=h_c(\julia(F_{c_0}))$, so there is a connected subset of $\julia(\Fc)$ joining $h_c(x)$ and $\partial V$. Therefore there exists $z\in\partial B(0,1/2)$ such that $\varphi(z)\in\julia(\Fc)$. Hence $\varphi(B(0,1/2))$ does not contains any ball centered at $h_c(x)$ and with radius greater than $|\varphi(z)-h_c(x)|$. So, if we set $\varphi(z)=h_c(w)$, with $w\in B(x,1)\cap \julia(F_{c_0})$ and $\varphi(z)\in V\cap\julia(\Fc)$ it follows from Proposition \ref{prop9.5} that 
\begin{align*}
|\varphi(z)-h_c(x)| & = |h_c(w)-h_c(x)| \\ 
	& \leq |h_c(w)-w| + |w-x| + |x-h_c(x)| \\ 
	& = |h_c(w)-h_c(x)| + 1 + |h_{c_0}(x)-h_c(x)| \\ 
	& = \left|\int_{c_0}^c h_\gamma'(w)d\gamma\right| + 1 + \left|\int_{c_0}^ch_\gamma(x)d\gamma\right| \\ 
	& \leq \int_{c_0}^c|h_\gamma'(w)||d\gamma| + 1 + \int_{c_0}^c|h_\gamma(x)||d\gamma| \\ 
	& \leq 1 + 2T|c-c_0|\leq 1 + 2rT. 
\end{align*}
Therefore $\varphi(B(0,1/2))$ does not contains the ball $B(h_c(x),2(1+2rT))$. $1/4$-Koebe's distortion theorem implies that 
\begin{equation}\label{eq9.9}
|\varphi'(0)|\leq 16(1+2rT).
\end{equation}
By considering the map $g=\varphi^{-1}\circ h_c: B(x,1)\to B(0,1)$ as a quasi-conformal homeomorphism between two disks of radius 1, Mori's theorem says that 
\[
|g(z_1)-g(z_2)|\leq16|z_1-z_2|^{1/q_c},
\]
for every $z_1,z_2\in B(x,1)$. In particular, for $z\in B(x,1)$ 
\[
|g(z)|\leq|g(z)-g(x)|\leq16|z-x|^{1/q_c}. 
\]
Therefore, $|z-x|\leq2^{-10}$ implies $|g(z)|\leq16(2^{-10})^{1/q_c}\leq16(2^{-10})^{1/2} = 1/2$. Then for $z_1,z_2\in B(x,2^{-10})$, from (\ref{eq9.9}) we have 
\begin{align*}
|h_c(z_2)-h_c(z_1)| & = |\varphi(g(z_2)) - \varphi(g(z_1))1\leq K|\varphi'(0)||g(z_2)-g(z_1)| \\ 
	& \leq 16K|\varphi'(0)||z_2-z_1|^{1/q_c}\leq16K(1+2rT)|z_2-z_1|^{1/q_c},
\end{align*}
where $K$ is the Koebe's distortion constant corresponding to the scale $1/2$, proving this way the Hölder regularity of the map $h_c$. 
\end{proof}

\section{Real Analyticity of the Hausdorff Dimension}\label{sec_an}
We consider first the following continuity result on the pressure function. 
\begin{lemma}\label{lemma10.1}
For $c\in\cD(\ell,1)$ and $t>1$, the function $(t,c)\mapsto P_c(t)$ is continuous. 
\end{lemma}
\begin{proof}
If we fix $c_0\in\cD(\ell,1)$, from the definition of $P(t)=\underline{P}(t)=\overline{P}(t)$ we know that $P_c(t)$ holds a Hölder inequality, which implies that $t\mapsto P_c(t)$ is continuous for $t\in(1,+\infty)$. Then, for $t_0>1$ there exists $\xi\in(0,t_0-1)$ such that if $t\in(t_0-\xi,t_0+\xi)$ 
\begin{equation}\label{eq10.1}
|P_{c_0}(t)-P_{c_0}(t_0)|<\varepsilon/2,
\end{equation}
for some $\varepsilon>0$. Now, consider $\gamma>1$ such that $(t_0+\xi)\log\gamma<\varepsilon/2$ and let $0<r_1\leq r$ be so small that 
\[
M = \inf\{\{\inf_{z\in\julia(\Fc)}\real(z)\}:c\in\cD(c_0,r_1)\}>-2\ell.
\]
There exists $0<r_2\leq r_1$ small enough that if $\real(z),\real(w)\geq M$, with $|z-w|\leq r_2$ then 
\[
\gamma^{-1}<\left|\dfrac{\ell-e^z}{\ell-e^w}\right|<\gamma.
\]
From Proposition \ref{prop9.3}-(c), there exists $\eta\in(0,\min\{\xi,r_2\})$ so that if $|c-c_0|<\eta$ then $|h_c(z)-z|<r_2$. Now fix $z\in\julia(F_{c_0})$ and take $n\geq1$ and $x\in F_{c_0}^{-n}(z)$. Since $h_c$ conjugates $F_c$ and $F_{c_0}$, we have $h_c(F_{c_0}^{-n}(z))=\Fc^{-n}(z)$. Moreover, for every $0\leq i\leq n$, and every $x\in F_{c_0}^{-n}(z)$, we have $h_c(F_{c_0}^i(x))=F_{c_0}^i(h_c(x))$, which implies $|F_{c_0}^i(x)-F_{c}^i(h_c(x))|\leq r_2$. Hence, 
\begin{align*}
\dfrac{|(F_c^n)'(h_c(x))|}{|(F_{c_0}^n)'(x)|} & = \dfrac{|(f_c^n)'(h_c(x))|}{|(f_{c_0}^n)'(x)|} = \dfrac{|\prod_{i=0}^{n-1}f_c'(f_c^i(h_c(x))|}{|\prod_{i=0}^{n-1}f_{c_0}'(f_{c_0}^i(x))|} \\ 
	& = \prod_{i=0}^n\dfrac{|\ell-\exp(f_c^i(h_c(x)))|}{|\ell-\exp(f_{c_0}^i(x))|} \in (\gamma^{-n},\gamma^n).
\end{align*}
Since $h_c:F_{c_0}^{-n}(z)\to F_c^{-n}(h_c(z))$ is a bijection, we conclude that 
\[
\dfrac{\sum_{x\in F_c^{-n}(h_c(z))}|(F_c^n)'(x)|^{-t}}{\sum_{x\in F_{c_0}^{-n}(z)}|(F_{c_0}^n)'(x)|^{-t}} \in (\gamma^{-tn},\gamma^{tn}), 
\]
so 
\begin{equation*}
\dfrac{1}{n}\log \left(\sum_{x\in F_c^{-n}(h_c(z))}|(F_c^n)'(x)|^{-t}\right)-\dfrac{1}{n}\log \left(\sum_{x\in F_{c_0}^{-n}(z)}|(F_{c_0}^n)'(x)|^{-t}\right) \in (-t\log \gamma,t\log\gamma)
\end{equation*}
This way, $P_c(t)-P_{c_0}(t)\in(-t\log\gamma,t\log\gamma)$ for every $c\in\cD(c_0,\eta)$ and consequently $|P_c(t)-P_{c_0}(t)|<\varepsilon/2$ for every $(t,c)\in(t_0-\eta,t_0+\eta)\times\cD(c_0,\eta)$. Combining (\ref{eq10.1}) we have 
\[
|P_c(t)-P_{c_0}(t_0)|<\varepsilon,
\]
for every $(t,c)\in (t_0-\eta,t_0+\eta)\times\cD(c_0,\eta)$ finishing the proof.
\end{proof}
In order to be able to use analytic property of transfer operators in Theorem \ref{cor8.7} we need the following notation. 

Fix $c_0\in\cD(\ell,1)$ and $t_0\in(1,+\infty)$. Following Proposition \ref{prop9.6} we have that $h_c:\julia(F_{c_0})\to\julia(\Fc)$ is $1/q_c$-Hólder continuous depending on $c$, and since $q_c$ converge to 1 as $c\to c_0$, we get that for $r>0$ sufficiently small 
\begin{equation}\label{eq10.2}
\alpha = \inf\{1/q_c:c\in\cD(c_0,r)\}>0.
\end{equation}

For $(t,c)\in(1,+\infty)\times\cD(c_0,r)$ let $\cL_{c,t}^0:\hh(\julia(F_{c_0}))\to\hh(\julia(\Fc))$ be the operator 
\[
\cL_{c,t}^0g(z) = \sum_{x\in F_{c_0}^{-1}(z)}|F_c'(h_c(x))|^{-t}g(x).
\]
Also, we need to embed $c\in\C^2$ as $(x+iy)\mapsto(x,y)\in\C^2$, and $t\in\C$ so we can use Theorem \ref{cor8.7}. The following is the last and most technical result in this work, and we left its proof to the end. 

\begin{proposition}\label{prop10.2}
Fix $(c_0,t_0)\in\cD(\ell,1)\times(1,+\infty)$. Then there exist $R>0$ and a holomorphic function 
\[
\fL:\D_{\C^3}((c_0,t_0),R)\to L(\hh(\julia(F_{c_0}))),
\]
($c_0\in\C^2,t_0\in\C$ and $\alpha$ comes from Equation (\ref{eq10.2}) with $r>0$ replaced by $R$) such that for every $(c,t)\in B(c_0,R)\times\cD(t_0,R)\subset\C\times\R$ 
\begin{equation}\label{eq10.3}
\fL(c,t)=\cL_{c,t}^0.
\end{equation}
\end{proposition}

\begin{theorem}
The function $c\mapsto\HD(\julia_r(\Fc))$, $c\in\cD(\ell,1)$ is real-analytic. 
\end{theorem}
\begin{proof}
Fix $c_0\in\cD(\ell,1)$ and $t_0>1$. From Proposition \ref{prop9.6} we know that $h_c:\julia(F_{c_0})\to\julia(\Fc)$ is $\alpha(c)$-Hölder continuous, $\alpha(c)$ depending on $c$, but converging to 1 as $c\to c_0$, so for every $r>0$ sufficiently small
\[
\alpha = \inf\{\alpha_c:c\in\cD(c_0,r)\}>0.
\]

For $(c,t)\in\cD(c_0,r)\times(1,+\infty)$, let 
\begin{equation}\label{eq10.22}
\fL(c,t)=\cL_{c,t}^0,
\end{equation}
be the (holomorphic) operator given by Proposition \ref{prop10.2}. From Theorem \ref{theo5.4} and Proposition \ref{prop_press} the value $e^{P_{c_0}(t)}$, $t\in B(t_0,R)$ is a simple isolated eigenvalue of the operator $\fL(c_0,t)=\cL_{c_0,t}^0:\hh(\julia(F_{c_0}))\to\hh(\julia(F_{c_0}))$. Applying perturbation theory for linear operators (see \cite{ka95} for further references), we see that there exists $0<R_1\leq R$ and a holomorphic function $\gamma:\D_{\C^3}((c_0,t),R_1)\to\C$ such that $\gamma(c_0,t_0)=e^{P_{c_0}(t_0)}$ and for every $(c,t)\in\D_{\C^3}((c_0,t_0),R_1)$ the number $\gamma(c,t)$ is a simple isolated eigenvalue of $\fL(c,t)$ with the reminder part of the spectrum uniformly separated from $\gamma(c,t)$. 

In particular, there exists $0<R_2\leq R_1$ and $\kappa>0$ such that 
\begin{equation}\label{eq10.23}
\sigma(\fL(c,t))\cap B(e^{P_{c_0}(t_0)},\kappa)=\{\gamma(c,t)\},
\end{equation}
for every $(c,t)\in\D_{\C^3}((c_0,t_0),R_2)$. Now, for each $(c,t)\in \cD(c_0,R)\times(t_0-R,t_0+R)$ consider the operator $\cL_{c,t}:\h1(\julia(F_c))\to\h1(\julia(F_c))$ given by 
\[
\cL_{c,t}g(z)=\sum_{x\in F_c^{-1}(z)}|F_c'(x)|^{-t}g(x).
\]
Take $T_c:\cC_b(\julia(\Fc))\to\cC_b(\julia(F_{c_0}))$, the map given by the formula $T_c(g)=g\circ h_c$. It is not difficult to see that conjugates $\cL_{c,t}:\cC_b(\julia(F_c))\to\cC_b(\julia(F_c))$ and $\cL_{c,t}^0:\cC_b(\julia(F_{c_0}))\to\cC_b(\julia(F_{c_0}))$. Given that $h_c:\julia(F_{c_0})\to\julia(F_c)$ is $\alpha$-Hölder continuous, we have 
\[
T_c(\h1(\julia(F_c)))\subset \hh(\julia(F_{c_0})). 
\]
Hence $e^{P_c(t)}$ is an eigenvalue of the operator
\[
\cL_{c,t}^0:\hh(\julia(F_{c_0}))\to\hh(\julia(F_{c_0})),
\]
and by Lemma \ref{lemma10.1}, $e^{P_c(t)}\in B(e^{P_c(t_0)},\kappa)$ for every $c\in\cD(c_0,R_3)$ and all $t\in(t_0-\rho,t_0+\rho)$ if $\rho\in(0,\min\{t_0,R_2\})$ and $R_3\in(0,R_2)$ are sufficiently small. Combining this and Equations (\ref{eq10.22}) and (\ref{eq10.23}) we see that $\gamma(c,t)=e^{P_c(t)}$ for every $c$ and $t$ as above. Therefore the function $(c,t)\mapsto P_c(t)$ with $(c,t)\in\cD(c_0,R_3)\times(t_0-\rho,t_0+\rho)$ is real-analytic. Since $P_c(s_c)=0$ for $s_c=\HD(\julia_r(F_c))$ by Theorem \ref{bowen}, it follows from the Implicit Function Theorem that to conclude the proof it is enough to show that 
\[
\dfrac{\partial P_c(t)}{\partial t}\neq 0,
\]
for all $(c,t)\in\cD(c_0,R_3)\times (t_0-\rho,t_0+\rho)$. So fix such $c$ and $t$. Fix $z\in\julia(F_c)$. Since for every $u\geq0$ and every $n\geq1$ 
\[
\sum_{x\in F_c^{-n}(z)}|(F_c^n)'(x)|^{-t+u}=\sum_{x\in F_c^{-n}(z)}|(F_c^n)'(x)|^{-t}|(F_c^n)'(x)|^{-u}\leq L^{-u}\beta^{un}\sum_{x\in F_c^{-n}(z)}|(F_c^n)'(x)|^{-t},
\]
we conclude that $P_c(t+u)-P_c(t)\leq u\log \beta$ which implies that $\dfrac{\partial P_c(t)}{\partial t}(c,t)\leq\log\beta<0$ and the proof is done.
\end{proof}

We finish this section with the proof of Lemma \ref{prop10.2}. 

\begin{proof}
(\textbf{of Proposition \ref{prop10.2}}.) Fix $c_0\in\cD(\ell,1)$ and $t_0>1$ and, to simplify notation, we fix $f=f_{c_0}$ and $F=F_{c_0}$. \\ 

For every $c\in\cU(c_0)$, a small neighbourhood of $c_0$, let $\theta_c=F_c'\circ h_c$ and for every $z\in\julia(F)$ let 
\begin{equation}\label{eq10.4-1}
\phi_{(c,t)}(z)=|\theta_c(z)|^{-t},
\end{equation}
and
\begin{equation}\label{eq10.4-2}
\psi_z(c)=\dfrac{\phi_c(z)}{\theta_{c_0}(z)},\ (c,z)\in\cU\times\julia(F). 
\end{equation}

Since 
\[
\lim_{z\to+\infty}\left(\dfrac{e^z}{\ell-e^{z}}\right)=-1,
\]
we see that 
\begin{equation}\label{eq10.6}
M:=\sup\left\{\left|\dfrac{e^z}{\ell-e^z}\right|:z\in\julia(\Fc),\ \real (z)>-2\ell\right\}<+\infty.
\end{equation}

Taking $r>0$ sufficiently small, it follows from Proposition \ref{prop9.5} that 
\begin{equation}\label{eq10.7}
|h_c(z)-z|=|h_c(z)-h_{c_0}(z)|=\left| \int_{c_0}^c h_\gamma'(z)d\gamma\right|\leq\int_{c_0}^c|h_\gamma'(z)||d\gamma|\leq T|c-c_0|\leq Tr. 
\end{equation}
Observe that 
\begin{equation}\label{eq10.8}
|1-e^w|\leq E|w|,
\end{equation}
for all $w\in B(0,Tr)$ and some $E>0$. Now, for every $z\in\julia(F)$, we have 
\[
\psi_z(c)=\dfrac{F_c'(h_c(z))}{F_{c_0}'(z)}=\dfrac{\ell-e^{h_c(z)}}{\ell-e^z} = \dfrac{\ell-e^z+e^z-e^{h_c(z)}}{\ell-e^z} = 1 + \dfrac{e^z-e^{h_c(z)}}{\ell-e^z}.
\]
Hence, using Equations (\ref{eq10.6}), (\ref{eq10.7}), and (\ref{eq10.8}) we obtain 
\[
|\psi_z(c)-1|=\left|\dfrac{e^z-e^{h_c(z)}}{\ell-e^z}\right|=\dfrac{|e^z|}{|\ell-e^z|}\cdot|1-e^{h_c(z)-z}|\leq ME|h_c(z)-z|\leq METr<1/2,
\]
where the last inequality was written assuming that $r>0$ is small enough. So there exists a well-defined branch of $\log\psi_z:\cU(c_0)\to\CC$ that satisfies
\begin{equation}\label{prop10.5}
|\log\psi_z(c)|\leq M_1, \text{ for some }M_1>0.
\end{equation}

Fix $\xi_1,\xi_2\in\julia(F)$, since $\julia(F)$ is connected, the segment $[\xi_1,\xi_2]$ joining the points $\xi_1$ and $\xi_2$ lies entirely in $\{z\in Q:\real (z)>-2\ell\}$, so from (\ref{eq10.6})
\begin{equation}\label{eq10.9}
|\log(e^{\xi_2}-\ell)-\log(e^{\xi_1}-\ell)| =\left|\int_{\xi_1}^{\xi_2}\dfrac{e^z}{\ell-e^z}dz\right|\leq\int_{\xi_1}^{\xi_2}\left|\dfrac{e^z}{\ell-e^z}\right||dz|\leq M|\xi_2-\xi_1|. 
\end{equation}

Fix now $z_1,z_2\in\julia(F)$ with $|z_2-z_2|\leq\delta$. Applying Proposition \ref{prop9.6}, (\ref{eq10.2}), and (\ref{eq10.9}), we have 
\begin{align*}
|\log\psi_{z_2}(c)-\log\psi_{z_1}(c)| & = |\log(F_c'\circ h_c(z_2))-\log(F_{c_0}'\circ h_{c_0}(z_2))-\log(F_c'\circ h_c(z_1)) + \log(F_{c_0}'\circ h_{c_0}(z_1))| \\
	& = |\log(\ell-e^{h_c(z_2)})-\log(\ell-e^{h_{c_0}(z_2)}) - \log(\ell-e^{h_{c_0}(z_1)}) + \log(\ell-e^{z_1})| \\ 
	& \leq |\log(\ell-e^{h_c(z_2)})-\log(\ell-e^{h_{c}(z_1)})| + |\log(\ell-e^{z_2})-\log(\ell-e^{z_1})| \\
	& \leq M|h_c(z_2)-h_c(z_1)| + M|z_2-z_1| \\
	& \leq M\hat{K}_{c_0}|z_2-z_1|^{\alpha}
\end{align*} 

Hence for every $c\in\cD(c_0,r)$, the map $z\mapsto\log\psi_z(c)$, $z\in\julia(F)$, belongs to $\hh$ and its Hölder constant is bounded from above by $M\hat{K}_{c_0}$. Since the function $\log\psi_z:\cD(c_0,r)\to\C$ is holomorphic, it is uniquely represented as a power series 
\[
\log\psi_z(c)=\sum_{n=0}^\infty a_n(z)(c-c_0)^n. 
\]
By Cauchy's inequality
\begin{equation}\label{eq10.10}
|a_n(z)|\leq\dfrac{M_1}{r^n},\ \forall n\geq0.
\end{equation}
For every $c=x+iy)\in\cD(c_r,r)\subset\C$, we have 
\begin{align}\label{eq10.11}
\real\log\psi_z(c) & = \real \left(\sum_n a_n(z)\left( (x-\real(c_0)) + i(y-\ima(c_0)\right)^n\right) \\
	& = \sum_{p,q=0}^\infty c_{p,q}(z)(x-\real(c_0))^p(y-\ima(c_0))^q, \notag
\end{align}
where $c_{p,q}(z)=a_{p+q}(z)\left(\begin{array}{c} p+q \\ q \end{array}\right)i^q$. From Equation (\ref{eq10.10})
\begin{equation}\label{eq10.12}
|c_{p,q}(z)|\leq|a_{p+q}(z)|\cdot 2^{p+q}\leq M_12^{p+q}r^{-(p+q)}. 
\end{equation}
Hence $\real\log\psi_z(c)$ extends by the same power series expansion (\ref{eq10.11}) to a holomorphic function on the polydisc $\D_{\C^2}(c_0,r/4)$. We denote this extension by the same symbol $\real\log\psi_z$ and we have 
\begin{equation}\label{eq10.13}
|\real\log\psi_z(c)|\leq \sum_{p,q=0}^\infty M_12^{-(p+q)} = 4M,
\end{equation}
on $\D_{\C^2}(c_0,r/4)$. So, for every $t\in\D_{\C}(t,\rho)$, where $\rho=t_0-1$, the formula 
\begin{equation}\label{eq10.14}
\zeta_{(c,t)}(z)=-(t\real\log\psi_z(c) + t\log|\theta_{c_0}(z)|), 
\end{equation}
extends $-t\log|\theta_c(z)|$ on the polydisc $\D_{\C^2}(c_0,r/4)\times\D_\C(t,\rho)$. Now, due to (\ref{eq10.13}), for every $(c,t)\in\D_{\C^2}(c,r/4)\times\D_\C(t,\rho)$ and every $z\in\julia(F)$, we have 
\begin{align}\label{eq10.15}
|e^{\zeta_{c,t}(z)}| & = \exp(\real(-t\real\log\psi_z(c)))|\theta_{c_0}(z)|^{-\real(t)} \notag\\
	& \leq \exp(|t||\real\log\psi_z(c)|)|\theta_{c_0}(z)|^{-\real(t)} \\
	& \leq e^{4M_1|t|}|\theta_{c_0}(z)|^{-\real(t)}. \notag
\end{align}
Since, by definition, $|\theta_{c_0}|^{-\real t}$ is summable, it therefore follows that each function 
\[
\phi_{(c,t)} = e^{\zeta_{(c,t)}}, \ (c,t)\in\D_{\C^2}(c_0,r/4)\times\D_\C(t_0,\rho),
\]
is summable. In order to prove the proposition we shall check that the family of functions $\phi_{(c,t)}$ satisfies all assumptions of Theorem \ref{cor8.7}. 

We define
\[
\fL(c,t)=\cL_{c,t}^0,
\]
so we satisfied Equation (\ref{eq10.3}).

\noindent \textbf{Claim:} The family of functions $\{\phi_{(c,t)}\}_{(c,t)\in G}$, with $G=\D_{\C^2}(c_0,r/4)\times\D_\C(t_0,\rho)\subset\C^3$ holds assumptions in Theorem \ref{cor8.7}. 


\noindent \textit{(a) $\phi_{(c,t)}$ is in $\hh^s$:} We already check that $\phi_{(c,t)}$ is summable. Recall that for every $c\in\cD(c_0,r)$ the map $z\mapsto\log\psi_z(c)$, $z\in\julia)(F)$, is in $\hh$ with Hölder constand bounded from above by $M\hat{K}_{c_0}$, from Cauchy's inequalities we have 
\[
|a_n(z)-an(w)|\leq M\hat{K}_{c_0}\left(\dfrac{4}{r}\right)^n|z-w|^\alpha,
\]
for all $z,w\in\julia(\Fc)$ with $|z-w|\leq\delta$, so 
\begin{equation}\label{eq10.16}
|c_{p,q}(z)-c_{p,q}(w)|\leq M\hat{K}_{c_0}2^{p+q}\left(\dfrac{4}{r}\right)^{p+q}|z-w|^\alpha = M\hat{K}_{c_0}\left(\dfrac{8}{r}\right)^{p+q}1z-w|^\alpha.
\end{equation}
Hence
\begin{equation}\label{eq10.17}
|\real \log\psi_z(c)-\real\log\psi_w(c)1\leq M\hat{K}_{c_0}\sum_{p,q=0}^\infty\left(\dfrac{8}{r}\right)^{p+q}\left(\dfrac{r}{16}\right)^p\left(\dfrac{r}{16}\right)^q|z-w|^\alpha=4M\hat{K}_{c_0}|z-w|^\alpha,
\end{equation}
for every $c\in\D_{\C^2}(c_0,r/16)$ and all $z,w\in\julia(F)$ with $|z-w|\leq\delta$. From Equation (\ref{eq10.13}) we know that $z\mapsto\real\log\psi_z(c)$, $z\in\julia(F)$ is in $\hh$ for every $c\in\D_{\C^2}(c_0,r/4)$. Since $\theta_{c_0}(z) = F_{c_0}'(z)$ we get that $\log|\theta_{c_0}(z)|=\log|\ell-e^z|$. Combining Equation (\ref{eq10.9}), (\ref{eq10.14}, and \ref{eq10.17}) we have 
\begin{align}\label{eq10.18}
|\zeta_{(c,t)}(z)-\zeta_{(c,t)}(w)| & = |t|(|\real\log\psi_z(c)-\real\log\psi_w(c)| -\log|\theta_{c_0}(z)| + \log|\theta_{c_0}(w)||) \notag \\ 
	& \leq (t_0+\rho)\left(4M\hat{K}_{c_0}|z-w|^\alpha + |\log|\theta_{c_0}(z)|-\log|\theta_{c_0}(w)||\right) \notag \\ 
	& = (t_0+\rho)\left(4M\hat{K}_{c_0}|z-w|^\alpha + \left|\log|e^z-\ell|-\log|e^w-\ell|\right|\right) \notag \\ 
	& = (t_0+\rho)\left(4M\hK_{c_0}|z-w|^\alpha + \left|\real\left(\log(e^z-\ell)-\log(e^w-\ell)\right)\right|\right)  \\
	& \leq (t_0+\rho)\left(4M\hK_{c_0}|z-w|^\alpha + \left|\log(e^z-\ell)-\log(e^w-\ell)\right|\right) \notag \\ 
	& \leq (t_0+\rho)\left(4M\hK_{c_0}|z-w|^\alpha + M|z-w|\right) \notag \\
	& \leq M(4\hK_{c_0}+1)(t_0+\rho)|z-w|^\alpha, \notag
\end{align}
for every $(c,t)\in\D_{\C^2}(c_0,r/16)\times\D_\C(t_0,\rho)$. Now, from Equation (\ref{eq10.15}) there exists $M_2>0$ such that  
\begin{equation}\label{eq10.19}
|\phi_{(c,t)}(z)|=|e^{\zeta_(c,t)(z)}|\leq M_2,
\end{equation}
for all ($(c,t)\in\D_{\C^2}(c_0,r/16)\times\D_\C(t_0,\rho)$ and all $z\in\julia(F)$. Let $M_3>0$ be such that $|e^\eta-1|<M_3|\eta|$ for all $\eta\in\C$ with $|\eta|\leq C\delta^\alpha$. Combining Equations (\ref{eq10.18}) and (\ref{eq10.18}), we obtain 
\begin{align*}
|\phi_{(c,t)}(z)-\phi_{(c,t)}(w)| & = |e^{\zeta_{(c,t)}(w)}|\cdot|e^{\zeta_{(c,t)}(z)-\zeta_{(c,t)}(w)}-1|\leq M_3M_2|\zeta_{(c,t)}(z)-\zeta_{(c,t)}(w)| \\ 
	& \leq CM_2M_3(|t_0|+\rho)|z-w|^\alpha,
\end{align*}
for all $(c,t)\in\D_{\C^2}(c_0,r/16)\times\D_{\C}(t_0,\rho)$ and all $z,w\in\julia(F)$ with $|z-w|\leq\delta$. In particular, we have that $\phi_{(c,t)}\in\hh$ verifying (a). 

\noindent \textit{(b) and (d) For every $(c,t)\in G$, the function $\phi_{(c,t)}$ is dynamically $\alpha$-Hölder with Hölder constants $c_{\phi_\alpha}$ uniformly bounded:} 

Fix $c\in\D_{\C^2}(c_0,r/16)$, $n\geq1$, $v\in F^{-n}(x)$ and $x,y\in\julia(F)$ with $|x-y|\leq\delta$. Applying Equations (\ref{eq1*}) and (\ref{eq10.18}) we have 
\begin{align*}
|\sum_{i=0}^{n-1}\zeta_{(c,t)}(F^i(F_v^{-n}(y)))-\sum_{i=0}^{n-1}\zeta_{(c,t)}(F^i(F_v^{-n}(x)))| & \leq \sum_{i=0}^{n-1}|\zeta_{(c,t)}(F^i(F_v^{-n}(y)))-\zeta_{(c,t)}(F^i(F_v^{-n}(x)))| \\ 
	& \leq \sum_{i=0}^{n-1}C(|t_0|+\rho)|F^i(F_v^{-n}(y))-F^i(F_v^{-n}(x))|^\alpha \\ 
	& \leq CL^{\alpha}(|t_0|+\rho)\sum_{i=0}^{\infty}\beta^{(n-i)\alpha}|y-x|^{\alpha} \\ 
	& \leq \dfrac{CL^\alpha}{1-\beta^{\alpha}}(|t_0|+\rho)|y-x|^\alpha. 
\end{align*}
By putting $M_4=\sup\left\{\left|\dfrac{e^{-z}-1}{z}\right|:|z|\leq\dfrac{CL^\alpha}{1-\beta^\alpha}(|t_0|+\rho)\right\}<\infty$, we obtain
\begin{align*}
D & = |\phi_{(c,t),n}(F_v^{-n}(y)) - \phi_{(c,t),n}(F_v^{-n}(x))| \\ 
	& \leq |\phi_{(c,t),n}(F_v^{-n}(x))|\left|\exp\left(\sum_{i=0}^{n-1}\zeta_{(c,t)}(F^i(F_v^{-n}(y))) - \sum_{i=0}^{n-1}\zeta_{(c,t)}(F^i(F_v^{-n}(x)))\right) - 1 \right| \\ 
	& \leq \dfrac{CM_4L^\alpha}{1-\beta^\alpha}(|t_0|+\rho)|\phi_{(c,t),n}(F_v^{-n}(x))|\cdot|y-x|^\alpha,
\end{align*}
proving simultaneously assumptions (b) and (d). 

\noindent \textit{(c) The function $(c,t)\mapsto\phi_{(c,t)}$ is continuous} 

Considering that 
\[
\phi_{(c,t)}(z) = e^{-t\real\log\psi_z(c)}\cdot|\theta_{c_0}(z)|^{-t}, 
\]
it is enough to show that both maps $z\mapsto e^{-t\real\log\psi_z(c)}$ and $z\mapsto|\theta_{c_0}(z)|^{-t}$ are in $\hh$ and that the maps 
\[
(c,t)\mapsto e^{-t\real\log\psi_{(\cdot)}(c)}\in\hh\quad\text{and}\qquad (c,t)\mapsto|\theta_{c_0}(\cdot)|^{-t}\in\hh,
\]
are continuous. We already know that the map $z\mapsto\real\log\psi_z(c)$, $z\in\julia(F)$, is in $\hh$ which implies that the map $z\mapsto t\real\log\psi_z(c)$, $z\in\julia(F)$, is in $\hh$ for every $t\in\R$. We aim to prove that $(c,t)\mapsto t\real\log\psi_{(\cdot)}(c)\in\hh$ is a continuous map on $\D_{\C^3}((c_0,t),R)$ for some $R>0$ sufficiently small. It is not difficult to see that such map is continuous with respect to the variable $t$ on $\D_{\C^2}(c_0,r/16)\times\D_\C(t_0,\rho)$. So it is enough to prove that $c\mapsto-t\real\log\psi_{(\cdot)}(c)\in\hh$ is Lipschitz continuous with Lipschitz constant independent of $t$. We fix $c=(c_x,c_y)$, $c'=(c_x',c_y')\in\D_{\C^2}(c_0,r/16)$. From Equation (\ref{eq10.11}), for all $z\in\julia(F)$ we have 
\begin{align}\label{eq10.20}
|\cdot| & = |\real\log\psi_z(c')-\real\log\psi_z(c)| \\ 
	& = \sum_{p,q=0}^\infty c_{p,q}(z)\left((c_x'-\real c_0)^p(c_y'-\ima c_0)^q - )c_x-\real c_0)^p(c_y-\ima  c_0)^q\right). \notag
\end{align}
Let $a_x=c_x'-\real c_0$, $a_y=c_y'-\ima c_0$, $b_x=c_x-\real c_0$, and $b_y=c_y-\ima c_0$. Then 
\begin{align}\label{eq10.21}
|a_x^pa_y^q - b_xb_y^q|	& = |a_x^p(a_y^q-b_y^q) + b_y^q(a_x^p-b_x^p)| \notag \\ 
	& = |a_x^p||a_y-b_y|\sum_{i=0}^{q-1}|a_y|^i|b_y|^{q-1-i} + |b_y^q||a_x-b_x|\sum_{i=0}^{p-1}|a_x|^i|b_x|^{´p-1-i} \\ 
	& \leq \left(q\left(\dfrac{r}{16}\right)^p\left(\dfrac{r}{16}\right)^{q-1} + p\left(\dfrac{r}{16}\right)^q\left(\dfrac{r}{16}\right)^{p-1}\right)\|c'-c\| \notag \\ 
	& \leq \dfrac{16}{r}(p+q)\left(\dfrac{r}{16}\right)^p\left(\dfrac{r}{16}\right)^q\|c'-c\|. \notag
\end{align}
Combining with Equations (\ref{eq10.12}) and (\ref{eq10.20}) we obtain 
\[
|\real\log\psi_z(c') - \real\log\psi_z(c)|\leq\dfrac{16}{r}\|c'-c\|\sum_{p,q=0}^\infty (+q)8^{-(p+q)} = \dfrac{16 C_1}{r}\|c'-c\|, 
\]
(where $C_1 = \sum_{p,q=0}^\infty (p+q)8^{-(p+q)}<+\infty$) for all $t\in\D_\C(t_0,\rho/2)$ and all $c,c'\in\D_{\C^2}(c_0,r/16)$. Now fix $z,w\in\julia(F)$ with $|z-w|\leq\delta$. It follows from (\ref{eq10.16}) and (\ref{eq10.21}) that 
\begin{align*}
|\cdot| & = |t|\left|\real\log \psi_w(c') - \real\log\psi_w(c) - (\real\log\psi_z(c') - \real\log\psi_z(c))\right| \\ 
	& = |t|\left|\sum_{p,q=0}\infty (c_{p,q}(w) - c_{p,q}(z))\left((c_x'-\real c_0)^p(c_y'-\ima c_0)^q - (c_x-\real c_0)^p(c_y-\ima c_0)^q)\right)\right| \\ 
	& \leq (|t_0|+\rho)\dfrac{32C}{r}|z-w|^\alpha\|c'-c\|\sum_{p,q=0}^\infty(p+q)2^{-(p+q)} \\ 
	& \leq \dfrac{4M\hat{K}_{c_0}C_2}{r}(|t_0|+\rho)\|c'-c\||z-w|^\alpha, 
\end{align*}
where $C_2=\sum_{p,q=0}^\infty(p+q)2^{-(p+q)}<+\infty$.  Which implies 
$$
v_\alpha\left(-t\real\log\psi_{(\cdot)}(c) - (-t\real\log\psi_{(\cdot)}(c'))\right) \leq 4M\hK_{c_0}C_2(|t_0|+\rho)r^{-1}\|c'-c\|.
$$
for all $c,c'\in\D_{\C^2}(c_0,r/16)$ and $t\in\D_\C(t_0,\rho)$, proving the continuity of the map $(c,t)\mapsto\psi_{(c,t)}(\cdot)=\exp\left(-t\real\log\psi_{(\cdot)}(c)\right)$, $(c,t)\in\D_{\C^2}(c_0,r/16)\times\D_\C(c_0,\rho)$. This, combining with Equation (\ref{eq10.13}) and the inequality $|e^b-e^a|=|e^a||e^{b-a}-1|\leq A|b-a|$, for some $A$ depending on $|a|$, we have the continuity of the function $(c,t)\mapsto\psi_{(c,t)}(\cdot)=\exp\left(-t\real\log\psi_{(\cdot)}(c))\right)$. 

Rest to prove that the function $z\mapsto|\theta_{c_0}(\cdot)|^{-t}$, $z\in\julia(F)$, is in $\hh$ for $(c,t)\in\D_{\C^2}(c_0,r/16)\times\D_\C(t_0,\rho)$. We know that 
$$
|\theta_{c_0}(z)|^{-t}=|\ell-e^z|^{-t}=\exp\left(-t\log|\ell-e^z|\right).
$$
It follows that 
$$
M_5=\sup \{|\theta_{c_0}(\cdot)|^{-t}:z\in\julia(F)\}<+\infty.
$$
Fixing additional $w\in\julia(F)\cap B(z,\delta)$, combining Equations (\ref{eq10.8}) and (\ref{eq10.9}) we get 
\begin{align*}
\left||\theta_{c_0}(z)|^{-t}-|\theta_{c_0}(w)|^{-t}\right| & = \left|\exp\left(-t\log|\ell-e^z|\right) - \exp\left(-t\log|\ell-e^z|\right)\right| \\ 
	& = |\ell-e^z|^{-t}|1 - \exp\left(t(\log|\ell-e^z| - \log|\ell-e^w|)\right)| \\ 
	& \leq M_5E\left|t(\log|\ell-e^z| - \log|\ell-e^z1)\right| \\ 
	& \leq M_5E(t_0\rho)M|z-w| \\ 
	& \leq M_5E(t_0+\rho)M|z-w|^\alpha. 
\end{align*}
Proving that the map is in $\hh$. Since $|\theta_{c_0}(z)|^{-t}$ does not depend on $c$, to prove its continuity it is enough to prove it only for $t$. 

We notice first that 
$$
\nabla(\ell-e^z) = -e^z(1,i),
$$
and
$$
\nabla(\overline{(\ell-e^z)} = \nabla(\ell-e^{\overline{z}})=-e^{\overline{z}}(1,-i) = -\overline{e^{{z}}}(1,-i).
$$
Using Leibnitz rule, 
$$
\nabla(|\ell-e^z|^2)=\nabla\left((\ell-e^z)\overline{(\ell-e^z)}\right) = -(\ell-e^z)\overline{e^{{z}}}(1,-i) - (\ell-\overline{e^{{z}}})e^z(1,i).
$$
Since $\real z\geq-2\ell$ we have 
$$
\|\nabla(|\ell-e^z|^2)\| \leq 2\ell|e^z\overline{e^{{z}}}|\|(1,-i)\|+2\ell|\overline{e^z}e^z|\|(1,i)\| = 4\sqrt{2}\ell|e^{2x}|. 
$$
Thus
$$
\|\nabla(|\ell-e^z|)\| = \|\nabla\left((|\ell-e^z|^2)^{1/2}\right)\|=\dfrac{1}{2}|\ell-e^z|^{-t}\|\nabla(|\ell-e^z|^2)\|\leq2\sqrt{2}\ell|e^{2x}||\ell-e^z|^{-1}.
$$
Since 
$$
\nabla\left(|\theta_{c_0}(z)|^{-t}\right) = \nabla\left(|\ell-e^z|^{-t}\right) = -t|\ell-e^z|^{-t-1}\nabla(|\ell-e^z|), 
$$
then, for all $t_1,t_2\in B(t_0,\rho)$ we have 
\begin{align*}
\lVert\nabla\left(|\theta_{c_0}(z)|^{-t_1} - |\theta_{c_0}(z)|^{-t_2}\right) \rVert & = |\ell - e^z|^{-1}\left| t_1|\ell-e^z|^{-t_1} - t_2|\ell-e^z|^{-t_2}\right|\|\nabla(|\ell-e^z|)\| \\
	& \leq 2\sqrt{2}\ell|e^{2x}||\ell-e^z|^{-2}\left||\ell-e^z|^{-t_1} - t_2|\ell-e^z|^{-t_2}\right|. 
\end{align*}
Since $(ta^{-t})'=a^{-t}(1-t\log a)$, using the Mean Value Inequality, we can continue the above estimates as follows 
\begin{align*}
|\cdot| & = \left\|\nabla\left(|\theta_{c_0}(z)|^{-t_1} - |\theta_{c_0}(z)|^{-t_2}\right)\right\| \\
	& \leq 2\sqrt{2}\ell|e^{2z}||\ell-e^z|^{-2}\sup\{|\ell-e^z|^{-t}(1-t\log|\ell-e^z|):t\in[t_1,t_2]\}|t_2-t_1|. 
\end{align*}
Since $\julia(F)\subset\{z\in Q:\real z\geq -2\ell\}$, we that 
$$
M_6:=\sup\left\{\left|\dfrac{e^z}{\ell-e^z}\right|:z\in\julia(F)\right\}<+\infty.
$$
Hence, from 
$$
M_7:=\sup\{|\ell-e^z|^{-t}(1-t\log|\ell-e^z|):z\in\julia(F),t\in[t_1,t_2]\}<+\infty,
$$
we obtain 
$$
\left|\nabla\left(|\theta_{c_0}(z)|^{-t_1} - |\theta_{c_0}(z)|^{-t_2}\right)\right|\leq2\sqrt{2}\ell M_6M_7|t_2-t_1|.
$$  
Then, for every $z\in\julia(F)$ and every $w\in B(z,\delta)$ we have 
\begin{align*}
\left||\theta_{c_0}(w)|^{-t_1}-|\theta_{c_0}(w)|^{-t_2} - \left(|\theta_{c_0}(z)^{-t_1}-|\theta_{c_0}(z)|^{-t_2}\right)\right| & \leq \int_z^w\left|\nabla\left(|\theta_{c_0}(\xi)|^{-t_1} - |\theta_{c_0}(\xi)|^{-t_2}\right)\right||d\xi| \\ 
	& \leq 2\sqrt{2}\ell M_6M_7|t_2-t_1||w-z| \\ 
	& \leq 2\sqrt{2}\ell M_6M_7|t_2-t_1||w-z|^{\alpha}.
\end{align*}
Again, with the Mean Value Inequality we obtain 
\begin{align*}
\left||\theta_{c_0}(z)|^{-t_1} - |\theta_{c_0}(z)|^{-t_2}\right| & \leq \log|\ell-e^z|\sup\{|\ell-e^z|:z\in\julia(F),t\in[t_1,t_2]\}|t_2-t_1| \\
	& \leq M_7|t_2-t_1|. 
\end{align*}
Which proves the continuity of the map $(c,t)\mapsto |\theta_{c_0}(\cdot)|^{-t}\in\hh$ completing the proof of assumption (c). 

\noindent \textit{(d) The function $(c,t)\mapsto\phi_{(c,t)}(z)\in C$, $(c,t)\in G$ is holomorphic for every $z\in\julia(F)$} 
This follows directly from the definition of $\phi_{(c,t)}$ and the first part of this proof. 

Rest to prove assumption (f):

$$\sup\left\{\left|\dfrac{\phi_{(c,t)}}{\phi_{(c_1,t_1)}}\right|:(c,t)\in\overline{B((c_2,t_2),r)}\right\}.$$
Fix $c_2\in\D_{\C^2}(c_0,r/16)$ and $t_2\in\D_\C(t_0,\rho)$. Take $\gamma>0$ so small that $\D_{\C^2}(c_2,\gamma)\subset\D_{\C^2}(c_0,r/16)$ and $\D_\C(t_2,2\gamma)\subset\D_\C(t_0,\rho)$. Then fix arbitrary $c_1\in\D_{\C^2}(c_0,r/16)$ and $t_1\in(1,\real(t_2)-\gamma)$. Then for every $(c,t)\in\D_{\C^2}(c_2,\gamma)\times\D_\C(t_2,\gamma)$, we have 
\begin{align*}
\left|\dfrac{\phi_{(c,t)}(z)}{\phi_{(c_1,t_1)}(z)}\right| & = \dfrac{e^{-t\real\log\psi_z(c)}}{e^{-t_1\real\log\psi_z(c_1)}}|\theta_{c_0}(z)|^{-(t-t_1)} = e^{t_1\real\log\psi_z(c_1) - t\real\log\psi_z(c)}\cdot|\theta_{c_0}(z)|^{-(t-t_1)} \\ 
	& = e^{t_1(\real\log\psi_(c_1)-\real\log\psi_z(c))}\cdot e^{(t_1-t)\real\log\psi_z(c)}\cdot|\theta_{c_0}(z)|^{-(t-t_1)}. 
\end{align*}
Using Equation (\ref{eq10.13}) we have the estimation 
$$
|e^{t_1(\real\log\psi_z(c_1)-\real\log\psi_z(c))}|=e^{\real(t_1(\real\log\psi_z(c_1)-\real\log\psi_z(c)))}\leq e^{|t_1(\real\log\psi_z(c_1)-\real\log\psi_z(c))|}\leq e^{8t_1M1},
$$
and
$$
|e^{(t_1-t)\real\log\psi_z(c)}|=e^{\real((t_1-t)\real\log\psi_z(c))}\leq e^{|(t_1-t)\real\log\psi_z(c)|}\leq e^{4\rho M_1}. 
$$
Since $A=\inf_{z\in\julia(F)}|\theta_{c_0}(z)|$ is positive, $\real(t_1-t)<0$ and $\real(t_1-t)>-\rho$, we have 
$$
|\theta_{c_0}(z)|^{-(t-t_1)}=|\theta_{c_0}(z)|^{\real(t_1-t)}\leq\min\{1,|\theta_{c_0}(z)|\}^{\real(t_1-t)}\leq\min\{1,|\theta_{c_0}(z)|\}^{-\rho}\leq\min\{1,A\}^{\rho}.
$$
Therefore 
$$
\left|\dfrac{\phi_{(c,t)}(z)}{\phi_{(c_1,t_2)}(z)}\right|\leq \exp\left(8t_1M_1+4\rho M_1\right)\min\{1,A\}^{-\rho},
$$
which verifies assumption (f) and concludes the proof. 

\end{proof}


\end{document}